\documentclass[12pt,a4paper]{amsart}

\usepackage{amsmath,amssymb,amsfonts,amsthm,graphicx,hyperref}
\usepackage{anysize}

 \newcommand\AGL{\mathrm{AGL}}   
 \newcommand\bbF{\mathbb{F}} 
      \newcommand\Cr{\mathrm{cr}}

 \newcommand\fix{\mathrm{fix}}
   
    \newcommand\GU{\mathrm{GU}}

 \newcommand\PG{\mathrm{PG}} \newcommand\PGL{\mathrm{PGL}} 
 \newcommand\PGaU{\mathrm{P\Gamma U}} \newcommand\PGU{\mathrm{PGU}}   \newcommand\PSL{\mathrm{PSL}}     \newcommand\PSU{\mathrm{PSU}}

  \newcommand\Ree{\mathrm{Ree}} 
       \newcommand\Sy{\mathrm{S}} \newcommand\Sym{\mathrm{Sym}} \newcommand\Sz{\mathrm{Sz}}

\newcommand\herm{\mathsf{H}}
\newcommand\F{\mathbb{F}}
\newcommand\graph{\mathsf{Graph}}

\newtheorem{theorem}{Theorem}[section]
\newtheorem{lemma}[theorem]{Lemma}
\newtheorem{proposition}[theorem]{Proposition}

\newtheorem{problem}[theorem]{Problem}

\theoremstyle{definition}

\begin{document}

\title{The covering radii of the $2$-transitive unitary, Suzuki, and Ree groups}
\author{John Bamberg, Cheryl E. Praeger, and Binzhou Xia}

\address[Bamberg]{
Centre for Mathematics of Symmetry and Computation\\
School of Mathematics and Statistics\\
University of Western Australia\\
35 Stirling Highway\\
Crawley 6009, Australia.\newline Email: {\tt john.bamberg@uwa.edu.au}
}

\address[Praeger]{
Centre for Mathematics of Symmetry and Computation\\
School of Mathematics and Statistics\\
University of Western Australia\\
35 Stirling Highway\\
Crawley 6009, Australia.\newline Email: {\tt cheryl.praeger@uwa.edu.au}
}

\address[Xia]{
School of Mathematics and Statistics\\
The University of Melbourne\\
Parkville, VIC 3010, Australia.\newline Email: {\tt binzhoux@unimelb.edu.au}
}

\maketitle

\begin{abstract}
We study the covering radii of $2$-transitive permutation groups of Lie rank one, giving bounds and links to finite geometry.


\end{abstract}

\section{Introduction}

Let $n\geqslant2$. Define the \emph{Hamming distance} $d_n$ on the symmetric group $\Sy_n$ by letting
\[
d_n(g,h)=n-|\fix(gh^{-1})|
\]
for any $g,h\in\Sy_n$, where $\fix(x)$ denotes the set of fixed points for $x\in\Sy_n$. Equipped with Hamming distance $\Sy_n$ is then a metric space. As usual, the distance of a point $v$ from a subset $C$ in $\Sy_n$ is
\[
d_n(v,C):=\min\{d_n(v,c)\mid c\in C\},
\]
and the \emph{covering radius} of $C$ is
\[
\Cr_n(C):=\max\{d_n(v,C)\mid v\in\Sy_n\}.
\]
Covering radii of subgroups of $\Sy_n$ were first studied in a seminal paper of Cameron and Wanless~\cite{CW2005}, where an upper bound was given in terms of the transitivity:

\begin{proposition}
\emph{(\cite[Proposition~16]{CW2005})} If $G$ is a $t$-transitive permutation group of degree $n$, then $\Cr_n(G)\leqslant n-t$.
\end{proposition}

In light of this result, the authors of~\cite{CW2005} were then interested in the covering radius of $2$-transitive permutation groups. For example, they showed that $\Cr_{q+1}(\PSL(2,q))=q-\gcd(2,q)$ for all $q$, and determined $\Cr_{q+1}(\AGL(2,q))$ and $\Cr_{q+1}(\PGL(2,q))$ for $q\not\equiv1\pmod{6}$. For $q\equiv1\pmod{6}$, a determination of $\Cr_{q+1}(\AGL(2,q))$ and $\Cr_{q+1}(\PGL(2,q))$ was accomplished in~\cite{WZ2013} and~\cite{Xia2017}, respectively. In determining the covering radii of these groups, the key was to find a large lower bound on the covering radius, which turns out to be ad hoc for each group. To shed light on the covering radii of $2$-transitive permutation groups, especially on establishing their lower bounds, we apply a general idea of ``field automorphism" construction to obtain lower bounds of the other $2$-transitive permutation groups that are simple groups of (twisted) Lie rank one: $\PSU(3,q)$, $\Sz(q)$ and $\Ree(q)$. This leads to the following main results of this paper, whose proof is at the end of Sections~\ref{sec1}, \ref{sec2} and~\ref{sec3}, respectively.

\begin{theorem}\label{BoundPSU}
Let $q=p^f$ with prime $p$. Then
\[
q^3-p\leqslant\Cr_{q^3+1}(\PGU(3,q))\leqslant\Cr_{q^3+1}(\PSU(3,q))\leqslant q^3-\gcd(2,q).
\]
In particular, if $q$ is even then
\[
\Cr_{q^3+1}(\PGU(3,q))=\Cr_{q^3+1}(\PSU(3,q))=q^3-2.
\]
Moreover, for the field automorphism $h$ defined in~\eqref{eq6}, $d_{q^3+1}(h,\PGU(3,q))=q^3-p$.
\end{theorem}

\begin{theorem}\label{BoundSz}
Let $q=2^{2m+1}$. Then $q^2-4\leqslant\Cr_{q^2+1}(\Sz(q))\leqslant q^2-2$. Moreover, for the field automorphism $h$ defined in~\eqref{eq7}, $d_{q^2+1}(h,\Sz(q))=q^2-4$.
\end{theorem}

\begin{theorem}\label{BoundRee}
Let $q=3^{2m+1}$. Then $q^3-27\leqslant\Cr_{q^3+1}(\Ree(q))\leqslant q^3-1$. Moreover, for the field automorphism $h$ defined in~\eqref{eq10}, $d_{q^3+1}(h,\Ree(q))=q^3-27$.
\end{theorem}

Cameron and Wanless~\cite{CW2005} mentioned a geometric interpretation for their work, namely a link between the covering radius of $\PGL(2, q)$ and the classical Minkowski planes. However they did not give many details. In Section~\ref{sec4} we explain this link in detail and then establish a similar link between the covering radius of $\PGU(3, q)$ and a certain finite geometry $\mathcal{U}_{2,2}$. We also pose the problem of characterising the classical ovoids of $\mathcal{U}_{2,2}$ via incidence geometry.

\section{Unitary groups}\label{sec1}

Let $q=p^f$ with $p$ prime, and $V$ be a $3$-dimensional vector space over $\bbF_{q^2}$ equipped with the unitary form
\begin{equation}\label{eq5}
x_1y_3^q+x_2y_2^q+x_3y_1^q
\end{equation}
for any vectors $(x_1,x_2,x_3)$ and $(y_1,y_2,y_3)$ in $V$. Denote by $\Omega$ the set of totally isotropic one-dimensional subspaces of $V$ and $\varphi$ the quotient map of $\GU(V)$ modulo its center. Then $|\Omega|=q^3+1$ and $\GU(V)^\varphi=\PGU(3,q)$ is a permutation group on $\Omega$. For any $\langle(x_1,x_2,x_3)\rangle\in V$, let
\begin{equation}\label{eq6}
\langle(x_1,x_2,x_3)\rangle^h=\langle(x_1^p,x_2^p,x_3^p)\rangle.
\end{equation}
Then $h$ is a permutation of $V$, and we have $\PGU(3,q)\rtimes\langle h\rangle=\PGaU(3,q)$.

\begin{lemma}\label{SystemPSU}
For each $\gamma\in\bbF_{q^2}^\times$ and $\delta\in\bbF_{q^2}$ such that $\delta^{q+1}=1$, the system of equations
\begin{equation}\label{eq3}
\begin{cases}
a^{p-1}=\gamma^{q+1}\\
b^p=b\gamma^q\delta\\
a+a^q+b^{q+1}=0
\end{cases}
\end{equation}
has at most $p-1$ solutions $(a,b)\in\bbF_{q^2}\times\bbF_{q^2}$.
\end{lemma}

\begin{proof}
We divide it into three cases by the value of $\gamma^{(q^2-1)/(p-1)}$.

\underline{Case~1.} $\gamma^{(q^2-1)/(p-1)}=-1$. Let $(a,b)\in\bbF_{q^2}\times\bbF_{q^2}$ be a solution of~\eqref{eq3}. Then by the first line of~\eqref{eq3} we have
\begin{eqnarray*}
a^q+a&=&a\left((a^{p-1})^{(q-1)/(p-1)}+1\right)\\
&=&a\left((\gamma^{q+1})^{(q-1)/(p-1)}+1\right)=a\left(\gamma^{(q^2-1)/(p-1)}+1\right)=0,
\end{eqnarray*}
which together with the third line of~\eqref{eq3} yields $b=0$. Moreover, the first line of~\eqref{eq3} has at most $p-1$ solutions for $a$. Hence there are at most $p-1$ such pairs $(a,b)$.

\underline{Case~2.} $\gamma^{(q^2-1)/(p-1)}=1$ and $\gamma^{(q^2-1)/(p-1)}\neq-1$. Note that this condition indicates $p>2$. Let $(a,b)\in\bbF_{q^2}\times\bbF_{q^2}$ be a solution of~\eqref{eq3}. Then by the first line and third line of~\eqref{eq3} we have $a\neq0$ (since $\gamma\neq0$) and
\begin{eqnarray}\label{eq4}
2a=a\left(\gamma^{(q^2-1)/(p-1)}+1\right)&=&a\left((\gamma^{q+1})^{(q-1)/(p-1)}+1\right)\\
&=&a\left((a^{p-1})^{(q-1)/(p-1)}+1\right)=a^q+a=-b^{q+1}.\nonumber
\end{eqnarray}
In particular, $b\neq0$. It then follows from the second line of~\eqref{eq3} that $b^{p-1}=\gamma^q\delta$, which has at most $p-1$ solutions for $b$. Thus, since $a$ is determined by $b$ as~\eqref{eq4} shows, we conclude that there are at most $p-1$ such $(a,b)$.

\underline{Case~3.} $\gamma^{(q^2-1)/(p-1)}\neq\pm1$. In this case, any solution $(a,b)\in\bbF_{q^2}\times\bbF_{q^2}$ of~\eqref{eq3} would satisfy (using the first and third lines of~\eqref{eq3})
\begin{eqnarray*}
b^{q+1}=-a(a^{q-1}+1)&=&-a\left((a^{p-1})^{(q-1)/(p-1)}+1\right)\\
&=&-a\left((\gamma^{q+1})^{(q-1)/(p-1)}+1\right)=-a\left(\gamma^{(q^2-1)/(p-1)}+1\right)\neq0
\end{eqnarray*}
and hence (using the second line of~\eqref{eq3})
\begin{eqnarray*}
1=b^{q^2-1}&=&(b^{p-1})^{(q^2-1)/(p-1)}\\
&=&(\gamma^q\delta)^{(q^2-1)/(p-1)}=\left(\gamma^{q^2+q}\delta^{q+1}\right)^{(q-1)/(p-1)}=(\gamma^{q+1})^{(q-1)/(p-1)},
\end{eqnarray*}
contrary to the condition $(\gamma^{q+1})^{(q-1)/(p-1)}=\gamma^{(q^2-1)/(p-1)}\neq1$. This completes the proof.
\end{proof}

\begin{lemma}\label{FixPSU}
For each $g\in\PGU(3,q)$ we have $|\fix(gh^{-1})|\leqslant p+1$, where $h$ is defined as in~\eqref{eq6}.
\end{lemma}

\begin{proof}
Suppose $g$ is an element of $\PGU(3,q)$ such that $gh^{-1}$ fixes at least $2$ points of $\Omega$, say $\langle u\rangle$ and $\langle v\rangle$. Note that $\langle e_1\rangle$ and $\langle e_3\rangle$ lie in $\Omega$ by~\eqref{eq5}, where $e_1=(1,0,0)$ and $e_3=(0,0,1)$. Since $\PGU(3,q)$ is $2$-transitive on $\Omega$, there exists $x\in\PGU(3,q)$ such that $\langle u\rangle^x=\langle e_1\rangle$ and $\langle v\rangle^x=\langle e_3\rangle$. Accordingly, $x^{-1}gh^{-1}x$ fixes $\langle e_1\rangle$ and $\langle e_3\rangle$.

Let $X=\langle\PGU(3,q),h\rangle=\PGU(3,q)\rtimes\langle h\rangle=\PGaU(3,q)$ and
\[
Y=\left\{
\begin{pmatrix}
\gamma^{q+1}&&\\
&\gamma^q\delta&\\
&&1
\end{pmatrix}
\ \middle|\ \gamma,\delta\in\bbF_{q^2}^\times,\ \delta^{q+1}=1\right\}.
\]
Observe that $Y$ preserves the form in~\eqref{eq5}. We have $Y^\varphi\leqslant\PGU(3,q)$ and $Y^\varphi\rtimes\langle h\rangle\leqslant X_{\langle e_1\rangle\langle e_3\rangle}$. Define a map
\[
\theta:(\gamma,\delta)\mapsto
\begin{pmatrix}
\gamma^{q+1}&&\\
&\gamma^q\delta&\\
&&1
\end{pmatrix}
^\varphi
\]
from $\bbF_{q^2}^\times\times\{\delta\in\bbF_{q^2}^\times\mid\delta^{q+1}=1\}$ to $Y^\varphi$. Then $\theta$ is a group empimorphism, and
\begin{align*}
|\ker(\theta)|&=|\{\gamma,\delta\in\bbF_{q^2}^\times\mid\gamma^{q+1}=1,\delta=\gamma^{-q},\delta^{q+1}=1\}|\\
&=\{\gamma\in\bbF_{q^2}^\times\mid\gamma^{q+1}=1\}=q+1.
\end{align*}
Consequently, $|Y^\varphi|=(q^2-1)(q+1)/|\ker(\theta)|=q^2-1$. Since $X$ is $2$-transitive on $\Omega$, it follows that
\[
|X_{\langle e_1\rangle\langle e_3\rangle}|=\frac{|X|}{|\Omega|(|\Omega|-1)}=\frac{|\PGaU(3,q)|}{(q^3+1)q^3}=2(q^2-1)f=|Y^\varphi\rtimes\langle h\rangle|.
\]
This implies that $X_{\langle e_1\rangle\langle e_3\rangle}=Y^\varphi\rtimes\langle h\rangle$.

Now as $x^{-1}gh^{-1}x\in X_{\langle e_1\rangle\langle e_3\rangle}$, there exists an integer $i$ such that
\[
x^{-1}gh^{-1}x=yh^i
\]
for some $y\in Y^\varphi$. Since $\PGaU(3,q)=\PGU(3,q)\rtimes\langle h\rangle$, we deduce that $h^i=h^{-1}$, taking both sides of the above equality modulo $\PGU(3,q)$. Hence $yh^{-1}=x^{-1}gh^{-1}x$ has the same number of fixed points as $gh^{-1}$. Write
\[
y=\begin{pmatrix}
\gamma^{q+1}&&\\
&\gamma^q\delta&\\
&&1
\end{pmatrix}
^\varphi
\]
with $\gamma\in\bbF_{q^2}^\times$, $\delta\in\bbF_{q^2}^\times$ and $\delta^{q+1}=1$. To complete the proof, we show that $|\fix(yh^{-1})\setminus\{\langle e_1\rangle,\langle e_3\rangle\}|\leqslant p-1$.

For any fixed point $\alpha=\langle(\alpha_1,\alpha_2,\alpha_3)\rangle$ of $yh^{-1}$ other than $\langle e_1\rangle$ and $\langle e_3\rangle$, we have
\begin{equation}\label{eq1}
\alpha_1\alpha_3^q+\alpha_2^{q+1}+\alpha_3\alpha_1^q=0
\end{equation}
by~\eqref{eq5}. If $\alpha_1=0$ or $\alpha_3=0$, then~\eqref{eq1} would imply $\alpha_2=0$, contrary to the assumption that $\alpha\neq\langle e_1\rangle$ or $\langle e_3\rangle$. Consequently, $\alpha=\langle(a,b,1)\rangle$ for some $(a,b)\in\bbF_{q^2}\times\bbF_{q^2}$ with $a\neq0$. It follows from~\eqref{eq5} that
\[
a+a^q+b^{q+1}=0,
\]
and
\begin{equation}\label{eq2}
\langle(a\gamma^{q+1},b\gamma^q\delta,1)\rangle=\langle(a^p,b^p,1)\rangle
\end{equation}
since $\alpha^y=\alpha^h$. Note that~\eqref{eq2} is equivalent to
\[
\begin{cases}
a^{p-1}=\gamma^{q+1}\\
b^p=b\gamma^q\delta\\
\end{cases}
\]
as $a\neq0$. We then conclude by Lemma~\ref{SystemPSU} that there are at most $p-1$ such $\alpha$, completing the proof.
\end{proof}

\begin{lemma}\label{FixPGU}
$d_{q^3+1}(h,\PGU(3,q))=q^3-p$, where $h$ is defined as in~\eqref{eq6}.
\end{lemma}

\begin{proof}
By Lemma~\ref{FixPSU}, $d_{q^3+1}(h,g)\geqslant q^3+1-|\fix(gh^{-1})|\geqslant q^3-p$ for each $g\in\PGU(3,q)$. Hence $d_{q^3+1}(h,\PGU(3,q))\geqslant q^3-p$. To complete the proof, we only need to show the existence of $g\in\PGU(3,q)$ such that $d_{q^3+1}(h,g)\leqslant q^3-p$, or equivalently, the existence of $g\in\PGU(3,q)$ such that $|\fix(hg^{-1})|\geqslant p+1$. Let $u=\langle(1,0,0)\rangle\in\Omega$ and $v=\langle(0,0,1)\rangle\in\Omega$. Note that any element of $\Omega$ other than $u$ and $v$ has form $\langle(a,b,1)\rangle$ for some $(a,b)\in\bbF_{q^2}\times\bbF_{q^2}$ with $a\neq0$ and $a+a^q+b^{q+1}=0$.

First assume that $p=2$. Then $u$, $v$ and $\langle(1,0,1)\rangle$ are all elements of $\Omega$ fixed by $h$. Hence $|\fix(h)|\geqslant3$, as desired.

Next assume that $p>2$. Let $\omega$ be a generator of $\bbF_{q^2}^\times$, and
\[
g=\begin{pmatrix}
\omega^{(q+1)(p-1)/2}&&\\
&\omega^{q(p-1)/2}&\\
&&1
\end{pmatrix}
^\varphi\in\PGU(3,q).
\]
Then it is straightforward to verify that $u$, $v$ and
\[
\langle(\omega^{i(q^2-1)/(p-1)+(q+1)/2},0,1)\rangle
\]
with $1\leqslant i\leqslant p-1$ are all elements of $\Omega$ fixed by $hg^{-1}$. Hence $|\fix(hg^{-1})|\geqslant p+1$, as desired. This completes the proof.
\end{proof}

\noindent\emph{Proof of Theorem~$\ref{BoundPSU}$.} Since $\PSU(3,q)$ is a $2$-transitive subgroup of $\Sy_{q^3+1}$, we deduce from~\cite[Proposition~16]{CW2005} that $\Cr_{q^3+1}(\PSU(3,q))\leqslant q^3-1$. Moreover, equality cannot be attained if $q$ is even according to~\cite[Theorem~21]{CW2005}. Hence
\begin{equation}\label{eq11}
\Cr_{q^3+1}(\PSU(3,q))\leqslant q^3-\gcd(2,q).
\end{equation}
Let $h$ be as defined in~\eqref{eq6}. Then
\[
\Cr(\PSU(3,q))\geqslant\Cr(\PGU(3,q))\geqslant d_{q^3+1}(h,\PGU(3,q)).
\]
This together with Lemma~\ref{FixPGU} and~\eqref{eq11} verifies Theorem~\ref{BoundPSU}.

\section{Suzuki groups}\label{sec2}

Let $q=2^{2m+1}$, $\ell=2^{m+1}$, and
\[
\Omega=\{(a,b,c)\in\bbF_q^3\mid c=ab+a^{\ell+2}+b^\ell\}\cup\{\infty\}.
\]
Clearly, $|\Omega|=q^2+1$. Let
\begin{equation}\label{eq7}
\infty^h=\infty\quad\text{and}\quad(a,b,c)^h=(a^2,b^2,c^2)\text{ for any }(a,b,c)\in\bbF_q^3.
\end{equation}
Then $h$ is a permutation of $\Omega$, and $X:=\Sz(q)\rtimes\langle h\rangle$ is a permutation group on $\Omega$.  We shall prove that $d_{q^2+1}(h,\Sz(q))\geqslant q^2-4$, which gives a lower bound for the covering radius of $\Sz(q)$.

\begin{lemma}\label{SystemSz}
For each $\kappa\in\bbF_q^\times$, the system of equations
\begin{equation}\label{eq9}
\begin{cases}
\kappa a=a^2\\
\kappa^{\ell+1}b=b^2\\
\kappa^{\ell+2}c=c^2\\
c=ab+a^{\ell+2}+b^\ell
\end{cases}
\end{equation}
has exactly four solutions $(a,b,c)\in\bbF_q^3$.
\end{lemma}

\begin{proof}
From the first three lines of~\eqref{eq9} we deduce that $a=0$ or $\kappa$, $b=0$ or $\kappa^{\ell+1}$ and $c=0$ or $\kappa^{\ell+2}$. This gives eight candidates for $(a,b,c)$. Verifying the fourth line of~\eqref{eq9} for these eight candidates we conclude that
\[
(0,0,0),\quad(0,\kappa^{\ell+1},\kappa^{\ell+2}),\quad(\kappa,0,\kappa^{\ell+2}),\quad(\kappa,\kappa^{\ell+1},\kappa^{\ell+2})
\]
are all solutions of~\eqref{eq9} and are the only ones.
\end{proof}

\begin{lemma}\label{FixSz}
$d_{q^2+1}(h,\Sz(q))=q^2-4$, where $h$ is defined as in~\eqref{eq7}.
\end{lemma}

\begin{proof}
Suppose $g$ is an element of $\Sz(q)$ such that $gh^{-1}$ fixes at least two points of $\Omega$, say $u$ and $v$. Let $o=(0,0,0)\in\Omega$. Since $\Sz(q)$ is $2$-transitive on $\Omega$, there exists $x\in\Sz(q)$ such that $u^x=\infty$ and $v^x=o$. Accordingly, $x^{-1}gh^{-1}x$ fixes $\infty$ and $o$.

For any $\kappa\in\bbF_q^\times$ denote $y_\kappa$ the permutation of $\Omega$ defined by
\[
\infty^{y_\kappa}=\infty\quad\text{and}\quad(a,b,c)^{y_\kappa}=(\kappa a,\kappa^{\ell+1}b,\kappa^{\ell+2}c),
\]
and let $Y=\{y_\kappa\mid\kappa\in\bbF_q^\times\}$. Then $Y$ form a subgroup of $X_{\infty o}$~\cite[Page~250]{DM1996}. Since $X$ is $2$-transitive on $\Omega$, it follows that
\[
|X_{\infty o}|=\frac{|X|}{|\Omega|(|\Omega|-1)}=\frac{|\Sz(q)|(2m+1)}{(q^2+1)q^2}=(q-1)(2m+1)=|Y\rtimes\langle h\rangle|.
\]
This implies that $X_{\infty o}=Y\rtimes\langle h\rangle$.

Now as $x^{-1}gh^{-1}x\in X_{\infty o}$, there exist an integer $i$ and $\kappa\in\bbF_q^\times$ such that
\[
x^{-1}gh^{-1}x=y_\kappa h^i.
\]
Since $X=\Sz(q)\rtimes\langle h\rangle$, we deduce that $h^i=h^{-1}$, taking both sides of the above equality modulo $\Sz(q)$. Hence $y_\kappa h^{-1}=x^{-1}gh^{-1}x$ has the same number of fixed points as $gh^{-1}$. For any fixed point $(a,b,c)$ of $y_\kappa h^{-1}$ other than $\infty$, we have
\[
c=ab+a^{\ell+2}+b^\ell
\]
since $(a,b,c)\in\Omega$, and
\[
\begin{cases}
\kappa a=a^2\\
\kappa^{\ell+1}b=b^2\\
\kappa^{\ell+2}c=c^2
\end{cases}
\]
since $\alpha^{y_\kappa}=\alpha^h$. It then follows from Lemma~\ref{SystemSz} that $|\fix(y_\kappa h^{-1})\setminus\{\infty\}|\leqslant4$, and so $|\fix(gh^{-1})|=|\fix(y_\kappa h^{-1})|\leqslant5$. This implies that $|\fix(gh^{-1})|\leqslant5$ for any $g\in\Sz(q)$, which yields
\[
d_{q^2+1}(h,\Sz(q))\geqslant(q^2+1)-5=q^2-4.
\]
Moreover, Lemma~\ref{SystemSz} implies that
\[
d_{q^2+1}(h,y_\kappa)=q^2+1-|\fix(y_\kappa h^{-1})|=q^2+1-(1+4)=q^2-4
\]
for any $\kappa\in\bbF_q^\times$. We then conclude that $d_{q^2+1}(h,\Sz(q))=q^2-4$, as the lemma asserts.
\end{proof}

\noindent\emph{Proof of Theorem~$\ref{BoundSz}$.} Since $\Sz(q)$ is a $2$-transitive subgroup of $\Sy_{q^2+1}$, we deduce from~\cite[Proposition~16]{CW2005} that $\Cr_{q^2+1}(\Sz(q))\leqslant q^2-1$. Moreover, the equality cannot be attained according to~\cite[Theorem~21]{CW2005}. Hence
\[
\Cr_{q^2+1}(\Sz(q))\leqslant q^2-2.
\]
This together with Lemma~\ref{FixSz} leads to Theorem~\ref{BoundSz}.

\section{Ree groups}\label{sec3}

Let $q=3^{2m+1}$, $\ell=3^{m+1}$, and
\[
\Omega=\{(a,b,c,\lambda_1(a,b,c),\lambda_2(a,b,c),\lambda_3(a,b,c))\mid(a,b,c)\in\bbF_q^3\}\cup\{\infty\},
\]
where
\begin{align*}
\lambda_1(a,b,c)&=a^2b-ac+b^\ell-a^{\ell+3},\\
\lambda_2(a,b,c)&=a^\ell b^\ell-c^\ell+ab^2+bc-a^{2\ell+3},\\
\lambda_3(a,b,c)&=ac^\ell-a^{\ell+1}b^\ell+a^{\ell+3}b+a^2b^2-b^{\ell+1}-c^2+a^{2\ell+4}.
\end{align*}
Clearly, $|\Omega|=q^3+1$. Let $\infty^h=\infty$ and
\begin{align}\label{eq10}
&(a,b,c,\lambda_1(a,b,c),\lambda_2(a,b,c),\lambda_3(a,b,c))^h\\
=&(a^3,b^3,c^3,\lambda_1(a^3,b^3,c^3),\lambda_2(a^3,b^3,c^3),\lambda_3(a^3,b^3,c^3))\nonumber
\end{align}
for any $(a,b,c)\in\bbF_q^3$. Then $h$ is a permutation of $\Omega$, and $X:=\Ree(q)\rtimes\langle h\rangle$ is a permutation group on $\Omega$.  We shall prove that $d_{q^3+1}(h,\Ree(q))\geqslant q^3-27$, which gives a lower bound for the covering radius of $\Ree(q)$.

\begin{lemma}\label{FixRee}
$d_{q^3+1}(h,\Ree(q))=q^3-27$, where $h$ is defined as in~\eqref{eq10}.
\end{lemma}

\begin{proof}
Suppose $g$ is an element of $\Ree(q)$ such that $gh^{-1}$ fixes at least two points of $\Omega$, say $u$ and $v$. Let $o=(0,0,0)\in\Omega$. Since $\Ree(q)$ is $2$-transitive on $\Omega$, there exists $x\in\Ree(q)$ such that $u^x=\infty$ and $v^x=o$. Accordingly, $x^{-1}gh^{-1}x$ fixes $\infty$ and $o$.

For any $\kappa\in\bbF_q^\times$ denote by $y_\kappa$ the permutation of $\Omega$ defined by $\infty^{y_\kappa}=\infty$ and
\begin{align*}
&(a,b,c,\lambda_1(a,b,c),\lambda_2(a,b,c),\lambda_3(a,b,c))^{y_\kappa}\\
=&(\kappa a,\kappa^{\ell+1}b,\kappa^{\ell+2}c,\lambda_1(\kappa a,\kappa^{\ell+1}b,\kappa^{\ell+2}c),\lambda_2(\kappa a,\kappa^{\ell+1}b,\kappa^{\ell+2}c),\lambda_3(\kappa a,\kappa^{\ell+1}b,\kappa^{\ell+2}c)),
\end{align*}
and let $Y=\{y_\kappa\mid\kappa\in\bbF_q^\times\}$. Then $Y$ forms a subgroup of $X_{\infty o}$~\cite[Page~251]{DM1996}. Since $X$ is $2$-transitive on $\Omega$, it follows that
\[
|X_{\infty o}|=\frac{|X|}{|\Omega|(|\Omega|-1)}=\frac{|\Ree(q)|(2m+1)}{(q^3+1)q^3}=(q-1)(2m+1)=|Y\rtimes\langle h\rangle|.
\]
This implies that $X_{\infty o}=Y\rtimes\langle h\rangle$.

Now as $x^{-1}gh^{-1}x\in X_{\infty o}$, there exist an integer $i$ and $\kappa\in\bbF_q^\times$ such that
\[
x^{-1}gh^{-1}x=y_\kappa h^i.
\]
Since $X=\Ree(q)\rtimes\langle h\rangle$, we deduce that $h^i=h^{-1}$, taking both sides of the above equality modulo $\Ree(q)$. Hence $y_\kappa h^{-1}=x^{-1}gh^{-1}x$ has the same number of fixed points as $gh^{-1}$. For any fixed point $(a,b,c,\lambda_1(a,b,c),\lambda_2(a,b,c),\lambda_3(a,b,c))$ of $y_\kappa h^{-1}$ other than $\infty$, we have
\begin{equation}\label{eq8}
\begin{cases}
\kappa a=a^3\\
\kappa^{\ell+1}b=b^3\\
\kappa^{\ell+2}c=c^3
\end{cases}
\end{equation}
since $\alpha^{y_\kappa}=\alpha^h$. As each line of~\eqref{eq8} has most three solutions, \eqref{eq8} has at most $27$ solutions $(a,b,c)\in\bbF_q^3$. It follows that $|\fix(y_\kappa h^{-1})\setminus\{\infty\}|\leqslant27$, and so $|\fix(gh^{-1})|=|\fix(y_\kappa h^{-1})|\leqslant28$. This implies that $|\fix(gh^{-1})|\leqslant28$ for any $g\in\Ree(q)$, which yields
\[
d_{q^3+1}(h,\Ree(q))\geqslant(q^3+1)-28=q^3-27.
\]
Moreover, \eqref{eq8} has exactly $27$ solutions $(a,b,c)\in\bbF_q^3$ when $\kappa=1$. Hence $y_1h^{-1}$ fixes exactly $28$ points. This implies that
\[
\Cr_{q^3+1}(h,y_1)=(q^3+1)-28=q^3-27.
\]
We then conclude that $d_{q^3+1}(h,\Ree(q))=q^3-27$, as the lemma asserts.
\end{proof}

\noindent\emph{Proof of Theorem~$\ref{BoundRee}$.} Since $\Ree(q)$ is a $2$-transitive subgroup of $\Sy_{q^3+1}$, we deduce from~\cite[Proposition~16]{CW2005} that \[
\Cr_{q^3+1}(\Ree(q))\leqslant q^3-1.
\]
This together with Lemma~\ref{FixRee} leads to Theorem~\ref{BoundRee}.

\section{Geometric interpretations for the covering radii of $\PGL(2,q)$ and $\PGU(3,q)$}\label{sec4}

Let us first describe in detail how the covering radius problem for $\PGL(2,q)$ relates to certain configurations
in an associated \emph{Minkowski plane} that was alluded to by Cameron and Wanless~\cite{CW2005}.
Consider the projective line $\PG(1,q)$ for which $\PGL(2,q)$ has its
natural sharply $3$-transitive action of degree $q+1$.
Consider the Segre variety $\mathcal{S}_{m,n}(q)$ lying in $\PG((m+1)(n+1)-1,q)$ obtained
by taking the simple tensors of pairs of homogeneous coordinates of points
of $\PG(m,q)$ and $\PG(n,q)$. So for example, $\mathcal{S}_{1,1}(q)$ is the \emph{hyperbolic quadric}
$\mathsf{Q}^+(3,q)$ in projective $3$-space, and $\mathcal{S}_{1,2}(q)$ is the \emph{Segre threefold}
in $\PG(5,q)$. We refer to \cite[Chapter 4]{HT2016} for more on finite Segre varieties and their properties.
Consider $\mathcal{S}_{1,1}(q)$. The \emph{Segre map} in this instance is given by:
\begin{align*}
\PG(1,q)\times\PG(1,q)&\to \mathcal{S}_{1,1}\\
\left((X_1,X_2),(X_1',X_2')\right)&\mapsto (X_1,X_2)\otimes (X_1',X_2')=\begin{bmatrix} X_1X_1' & X_1X_2'\\ X_2X_1'& X_2X_2' \end{bmatrix}.
\end{align*}
We can identify the output with a row vector $(X_1X_1',X_1X_2',X_2X_1',X_2X_2')$;
the homogeneous coordinates for a point of $\PG(3,q)$.
Moreover, each element of $\mathcal{S}_{1,1,}$ is a zero of the following quadratic form on $\F_q^4$:
\[
Q(X)=X_1X_4-X_2X_3, \quad\text{for }X = (X_1,X_2,X_3,X_4)\in \F_q^4.
\]
The quadratic form $Q$ gives rise to a hyperbolic quadric $\mathsf{Q}^+(3,q)$ whose points are precisely those of $\mathcal{S}_{1,1}$.
Consider a permutation $\pi$ in $\Sym(\PG(1,q))$. Then we can identify
$\pi$ with a set of ordered pairs in $\PG(1,q)\times\PG(1,q)$:
\[
\pi\leftrightarrow \{ (X, X^\pi) \colon X \in \PG(1,q)\}.
\]
We can then apply the Segre map above to obtain
\[
\pi \rightarrow \graph(\pi):=\{ X \otimes X^\pi \colon X\in \PG(1,q)\}.
\]
Geometrically, $\graph(\pi)$ is an \emph{ovoid} of $\mathcal{S}_{1,1}$ because two
points $X\otimes X^\pi$ and $Y\otimes Y^\pi$ are orthogonal under the associated bilinear
form\footnote{In terms of tensors, the bilinear form can be written as $B(u_1\otimes u_2, v_1\otimes v_2)=
\beta(u_1,v_1)\beta(u_2,v_2)$, where $\beta$ is defined by $\beta(x,y)=x_1y_2-x_2y_1$, and then extend $B$ linearly.
So $0=B(X\otimes X^\pi, Y\otimes Y^\pi)=\beta(X,Y)\beta(X^\pi,Y^\pi)$ implies $X=Y$ or $X^\pi=Y^\pi$,
which are equivalent as $\pi$ is a bijection.}
if and only if $X=Y$. In other words, $\graph(\pi)$ is a set of $q+1$ points, no two orthogonal in
$\mathcal{S}_{1,1}$.
Now we can easily distinguish the permutations that lie in $\PGL(2,q)$. Recall that elements of $\PGL(2,q)$
are M\"obius transformations of the form
\[
\mu_{a,b,c,d}:=t\mapsto \frac{a t+b}{ct +d},\quad a,b,c,d\in\F_q, ad-bc\ne 0.
\]
So
\begin{align*}
\graph(\mu_{a,b,c,d})&=\{ (ct+d,at+b,ct^2+dt,at^2+bt) \colon t\in \F_q\cup \{\infty\} \},
\end{align*}
which is a conic section\footnote{In fact, the (hyper)plane of intersection has dual homogeneous coodinates
$[b, -d, a, -c]$ and it is not difficult to calculate that this plane is non-degenerate with respect to the form $Q$.
} of $\mathsf{Q}^+(3,q)$.
The converse is also true.
So in the bijection above, we have that $\PGL(2,q)$ is distinguished amongst all permutations
by taking conics amongst all ovoids of $\mathcal{S}_{1,1}$. Therefore, we have
\[
\Cr_{q+1}(\PGL(2,q))=\max_{ \mathcal{O}\in\text{ovoids of }\mathsf{Q}^+(3,q)}\left(
\min_ {\mathcal{C}\in\text{conics of }\mathsf{Q}^+(3,q)} |\mathcal{O}\cap \mathcal{C}|\right).
\]
Alternatively, we could use the language of \emph{Benz planes}. The \emph{classical Minkowski plane}
 $\mathcal{M}(q)$ consists of three types of objects: \emph{points}, \emph{lines}, and \emph{circles}. The points
 of $\mathcal{M}(q)$ are just the points of $\mathsf{Q}^+(3,q)$, the lines are the generators of $\mathsf{Q}^+(3,q)$,
 and the circles are the non-tangent plane sections (i.e., conics) of $\mathsf{Q}^+(3,q)$. Incidence is inherited
 from the ambient projective space. So to compute $\Cr_{q+1}(\PGL(2,q))$, we need to consider sets of points
 of $\mathcal{M}(q)$ that meet every line once and every circle in at most $s$ points, and we seek to minimise $s$.

We will consider a special set of points lying in the Segre variety $\mathcal{S}_{2,2}(q^2)\subseteq \PG(8,q^2)$.
Suppose we have an Hermitian form $\beta$ from $\F_{q^2}^3$ to $\F_{q^2}$. Define
$\beta\otimes \beta$ to be the form on $\F_{q^2}^9$ defined by
\[
(\beta\otimes \beta)(u_1\otimes v_1,u_2\otimes v_2) := \beta(u_1,u_2)\beta(v_1,v_2)
\]
and extend linearly. This form is Hermitian. The set of totally isotropic points of $\beta\otimes\beta$ is a Hermitian variety which we will denote
by $\herm(8,q^2)$.

Now consider the totally isotropic points of $\beta$; an \emph{Hermitian curve} $\herm(2,q^2)$ of $q^3+1$ points
of $\PG(2,q^2)$. This variety gives us the natural $2$-transitive action of $\PGU(3,q)$.
This time, we apply the Segre map to pairs of elements of the Hermitian curve:
\begin{align*}
\herm(2,q^2)\times\herm(2,q^2)&\to \mathcal{S}_{2,2}\\
\left(u,v\right)&\mapsto u\otimes v.
\end{align*}
This map is injective, but not surjective. There is no natural name for its image, but we will call it $\mathcal{U}_{2,2}$.

\begin{lemma}
Let $(P_1,P_2)\in \herm(2,q^2)\times\herm(2,q^2)$. Then the set of points orthogonal with or equal to
$P_1\otimes P_2$ (in $\mathcal{U}_{2,2}$) is
\[
\{ P_1\otimes X \colon X\in \herm(2,q^2) \} \cup \{ X\otimes P_2 \colon X\in \herm(2,q^2) \}.
\]
\end{lemma}

\begin{proof}
Suppose $X\otimes Y$ is orthogonal to $P_1\otimes P_2$ with respect to $\beta\otimes \beta$.
Then
\begin{align*}
\beta(X,P_1)\beta(Y,P_2)&=(\beta\otimes \beta)(X\otimes Y, P_1\otimes P_2)=0.
\end{align*}
However, no two distinct points of the Hermitian curve are orthogonal, and so we must have
either $X=P_1$ or $Y=P_2$.
\end{proof}

Now consider a permutation $f\in \Sym(\herm(2,q^2))$. Let $\graph(f)$ be
\[
\{ X\otimes X^f \colon X \in \herm(2,q^2) \}
\]
as a subset of $\mathcal{U}_{2,2}$. Then:

\begin{proposition}
For each $f\in \Sym(\herm(2,q^2))$, the set $\graph(f)$ is a set of $q^3+1$ totally isotropic points of $\herm(8,q^2)$,
no two distinct elements orthogonal in $\herm(8,q^2)$.
\end{proposition}

\begin{proof}
Suppose two elements $P\otimes P^f$ and $Q\otimes Q^f$ are orthogonal with respect to the form $\beta\otimes \beta$.
Then
\[
\beta(P,Q)\beta(P^f,Q^f)=(\beta\otimes \beta)(P\otimes P^f, Q\otimes Q^f)=0.
\]
and hence $P=Q$.
\end{proof}

The converse is also true.

\begin{proposition}
Let $\{ P_i \otimes Q_i \colon i=0,\ldots, q^3\}$ be a set of points of $\mathcal{U}_{2,2}$, no two distinct elements orthogonal.
Then the map $P_i \to Q_i$ is a bijection.
\end{proposition}

\begin{proof}
Suppose $Q_i=Q_j$ for some $i,j \in \{0,\ldots, q^3\}$. Then
\[
(\beta\otimes\beta)( P_i\otimes Q_i, P_j\otimes Q_j)=\beta(P_i,P_j) \beta(Q_i,Q_j)=0
\]
and hence $i=j$. Therefore, the map $P_i\to Q_i$ is a bijection.
\end{proof}

So we define an \emph{ovoid} of $\mathcal{U}_{2,2}$ to be a set of $q^3+1$ elements, that are pairwise non-orthogonal.
For example, the Frobenius automorphism $h$ that we used in the proof of Theorem~\ref{BoundPSU}
gives rise to the ovoid $\{ P\otimes P^h\colon P\in\herm(2,q^2) \}$ of $\mathcal{U}_{2,2}$.

\begin{proposition}
For each $f\in \PGU(3,q)$, the set $\graph(f)$ spans a projective $5$-space of $\PG(8,q^2)$, and its perp with respect to $\herm(8,q^2)$
is a plane $\pi_f$ which does not intersect $\mathcal{S}_{2,2}$. Moreover, $\pi_f$ is totally isotropic when $q$ is even, and non-degenerate
when $q$ is odd.
\end{proposition}

\begin{proof}
The diagonal $D:=\{ P\otimes P\colon P\in \PG(2,q^2)\}$ pertains to the projective points arising from
the symmetric tensors of the vector space $V:=\F_{q^2}^3$. The dimension of the symmetric square
of $\F_{q^2}^3$ is ${4\choose 2}=6$, and therefore, $D$ spans a projective $5$-space.
Moreover, as the Hermitian curve spans $\PG(2,q^2)$, we also know that $\graph(1)$ spans the same $5$-space,
where $1$ denotes the trivial element of $\PGU(3,q)$.
Now $\PGU(3,q)$ acts on $\{\graph(f)\colon f\in\PGU(3,q)\}$ in the following way:
\[
\graph(f)^g=\graph(fg).
\]
Note that since $\PGU(3,q)$ acts transitively on itself by right multiplication, we see that $\PGU(3,q)$
also acts transitively on $\{\graph(f)\colon f\in\PGU(3,q)\}$. So every $\graph(f)$ will span a space projectively equivalent
to the space spanned by $\graph(1)$.

Now let us consider the perp $\pi$ of $\graph(1)$.
Take the alternating tensors $A$ and construct projective points:
\[
A=\langle X\otimes Y - Y\otimes X\colon X,Y\in \PG(2,q^2) \rangle.
\]
Then for all $P,X,Y\in \PG(2,q^2)$ we have
\begin{align*}
(\beta\otimes \beta)( X\otimes Y - Y\otimes X, P\otimes P)&=(\beta\otimes \beta)( X\otimes Y, P\otimes P) -(\beta\otimes \beta)(Y\otimes X, P\otimes P) \\
&=\beta(X,P)\beta(Y,P)-\beta(Y,P)\beta(X,P)\\
&=0.
\end{align*}
Therefore, $A\subseteq \pi$. However, we know that $A$ has (algebraic) dimension ${3\choose 2}=3$ and hence $A=\pi$.
If $q$ is even, then $A$ is a subspace of $D$, and hence totally isotropic. Otherwise, $A\cap D$ is trivial and $A$
is non-degenerate.
\end{proof}

In analogy with the $\PGL(2,q)$ example from before, we will
call $\graph(f)$, where $f\in \PGL(3,q)$, a \emph{classical ovoid} of $\mathcal{U}_{2,2}$.
So in the injection above, we have that $\PGU(3,q)$ is distinguished amongst all permutations
of $\herm(2,q^2)$ by taking classical ovoids amongst all ovoids of $\mathcal{U}_{2,2}$.

\begin{problem}
Find an incidence geometry characterisation of classical ovoids of $\mathcal{U}_{2,2}$.
\end{problem}

We would then have a meaningful way to determine the following:
\[
\Cr_{q^3+1}(\PGU(3,q))=\max_{ \mathcal{O}\in\text{ovoids of }\mathcal{U}_{2,2}}\left(
\min_{\mathcal{C}\in\text{classical ovoids of }\mathcal{U}_{2,2}} |\mathcal{O}\cap \mathcal{C}|\right).
\]

\section*{Acknowledgments}

The third author's work on this paper was done when he was a research associate at the University of Western Australia supported by the Australian Research Council Discovery Project DP150101066. This work was inspired by the 2017 Centre for the Mathematics of Symmetry and Computation Research Retreat.

\end{document}